\documentclass[11pt]{article}
\usepackage[a4paper,margin=25mm]{geometry}
\usepackage{amsmath,amssymb,amsthm,booktabs,hyperref}
\hypersetup{hidelinks}
\newtheorem{theorem}{Theorem}
\newtheorem{proposition}[theorem]{Proposition}
\newtheorem{corollary}[theorem]{Corollary}
\newtheorem{lemma}[theorem]{Lemma}
\theoremstyle{remark}\newtheorem{remark}[theorem]{Remark}
\newcommand{\T}{\mathcal T_\alpha}
\newcommand{\E}{\mathcal E_\alpha}
\newcommand{\norm}[1]{\left\|#1\right\|_*}
\makeatother
\begin{document}
\title{Local energy and vertex deletion for $A_\alpha$-matrices}
\author{Ravinder Kumar and Piyush Sharma\\[1ex]
{\small Department of Mathematics and Computing,}\\
{\small Dr B R Ambedkar National Institute of Technology Jalandhar,}\\
{\small Punjab--144008, India.}\\[1ex]
{\small\texttt{thakurrk@nitj.ac.in}; \texttt{piyushs.mc.24@nitj.ac.in}}}
\maketitle 	
\begin{abstract}
We study the change in the trace norm and the centred energy of the $A_\alpha$-matrix of a graph under vertex deletion. A pinching argument proves that the trace norm is nondecreasing under passage from an induced subgraph to the original graph, and gives a lower bound in terms of the edges induced by the deleted vertices. Exact perturbation spectra yield sharp upper bounds for deleting a vertex or an independent set of vertices with a common neighbourhood. For $0<\alpha<1/2$, equality in these upper bounds is characterized by star and complete bipartite components, respectively. For the centred energy, we derive explicit deletion bounds and classify all connected equality graphs for $0\le\alpha<1$: the extremal graphs are stars and complete bipartite graphs for single-vertex and independent-twin deletion, respectively. We also characterize equality in the general-set trace-norm lower bound and prove strictness of the total centred-energy inequality for connected nonregular graphs when $0<\alpha<1$. Examples show that centred local energy can be negative and that the degree deviation term cannot be omitted in general.
\end{abstract}

\section{Introduction and preliminaries}
The energy of a graph is the sum of the absolute values of the eigenvalues of its adjacency matrix. The effect of graph operations on this quantity has been studied using singular-value inequalities and trace-norm methods; see, for example, Day and So \cite{DS} and the monograph \cite{LSG}. In particular, vertex deletion leads to a natural spectral measure of the contribution of a vertex to the energy of a graph.

Espinal and Rada \cite{ER} introduced deletion-based local energy and proved that deleting a vertex of degree $d$ decreases the adjacency energy by at most $2\sqrt d$. They also characterized equality by the presence of a star component whose centre is the deleted vertex. Rakshith and Das \cite{RD} extended this bound to an independent set of $k$ vertices with a common neighbourhood of size $d$, obtaining the upper bound $2\sqrt{kd}$. They further investigated total local energy and complete multipartite graphs. Ye \cite{Ye} subsequently obtained inequalities involving the energies of all $k$-vertex-deleted subgraphs by summing embedded adjacency matrices.

The matrices $A_\alpha(G)=\alpha D(G)+(1-\alpha)A(G)$, introduced by Nikiforov \cite{N}, interpolate between the adjacency and degree matrices. Two related spectral quantities must be distinguished: the trace norm of $A_\alpha(G)$ and its energy centred at the average eigenvalue. The latter is used in the $A_\alpha$-energy literature; see Zhou et al.\ \cite{ZLCM}. Deleting a vertex changes both the degree entries and the average eigenvalue, so neither quantity can be treated by simply replacing the adjacency matrix in an existing vertex-deletion argument.

Moon and Park \cite{MP} recently studied deletion-based local Laplacian energy. They derived a single-vertex deletion bound with a star equality characterization. Their total-energy estimate contains a degree-deviation term, reflecting the change in the centring constant under vertex deletion. Our centred total-energy bound is the corresponding estimate for $A_\alpha$-matrices. For bipartite graphs, its specialization at $\alpha=1/2$ agrees, after scaling, with their Laplacian result. The contributions here include bounds for the absolute centred-energy difference, independent-twin deletion, a connected equality classification for the full range $0\le\alpha<1$, and strictness of the total bound for connected nonregular graphs when $0<\alpha<1$. The perturbation formulas themselves follow from standard trace-norm inequalities and explicit quotient spectra.

In this paper, we first prove a deletion monotonicity theorem for the $A_\alpha$ trace norm and derive sharp upper bounds with structural equality conditions. We then compute the centred perturbation spectrum for single vertices and independent twins and classify its connected equality graphs, including singular perturbations. The general-set lower bound has an equality characterization by component separation and positive semidefiniteness. Finally, we establish the total local-energy estimates, prove strictness for connected nonregular graphs, and give degree-based equality criteria at $\alpha=1$. The degree-deviation term and the centring convention are essential to this distinction.

Let $G$ be a finite simple undirected graph of order $n\ge1$ and size $m$. Write $d_v$ for the degree of $v$, $N_G(v)$ for its open neighbourhood, and $G-S$ for the subgraph induced by $V(G)\setminus S$. For $0\le\alpha\le1$, put
\begin{align}
 A_\alpha(G)&=\alpha D(G)+(1-\alpha)A(G),\label{eq:defM}\\
 \T(G)&=\norm{A_\alpha(G)},\label{eq:defT}\\
 c_G&=\frac{2\alpha m}{n},\qquad C_\alpha(G)=A_\alpha(G)-c_GI_n,\label{eq:defC}\\
 \E(G)&=\norm{C_\alpha(G)}.\label{eq:defE}
\end{align}
Here $\norm X$ denotes the trace norm. For a real symmetric matrix, this is the sum of the absolute eigenvalues. We refer to $\T(G)$ as the $A_\alpha$ trace norm and to $\E(G)$ as the centred $A_\alpha$-energy. We assign zero trace norm and zero energy to the empty graph. At $\alpha=0$, both quantities equal the adjacency energy.

For a vertex $v$, define
\[
 L^T_\alpha(G;v)=\T(G)-\T(G-v),\qquad
 L^C_\alpha(G;v)=\E(G)-\E(G-v),
\]
and
\[
 \ell^T_\alpha(G)=\sum_{v\in V(G)}L^T_\alpha(G;v),\qquad
 \ell^C_\alpha(G)=\sum_{v\in V(G)}L^C_\alpha(G;v).
\]
These definitions concern energy changes under deletion and differ from vertex energy defined by diagonal entries of the absolute value of a graph matrix. We use $e(S)$ for the number of edges of $G[S]$ and $e(S,T)$ for the number of edges between disjoint vertex sets $S,T$. We also write $Q(G)=D(G)+A(G)$ and $L(G)=D(G)-A(G)$ for the signless Laplacian and Laplacian matrices, respectively.

We shall use the following elementary matrix facts. Their proofs are included to make the equality arguments self-contained.

\begin{lemma}\label{lem:embed}
Let $B_v$ be $A_\alpha(G-v)$ embedded at the retained vertices with a zero row and column at $v$. Then $H_v=A_\alpha(G)-B_v$ is supported on $v$ and its neighbours and is the embedded $A_\alpha$-matrix of $K_{1,d_v}$. Moreover,
\begin{equation}\label{eq:sum}
 \sum_{v\in V(G)}B_v=(n-2)A_\alpha(G),\qquad n\ge2.
\end{equation}
\end{lemma}
\begin{proof}
The diagonal entry of $H_v$ at $v$ is $\alpha d_v$, each neighbour has diagonal entry $\alpha$, and its only off-diagonal entries are $1-\alpha$ along edges incident with $v$. For the summation identity, an off-diagonal edge entry survives in $n-2$ embedded deletions. At a fixed vertex $u$, the surviving degrees sum to $(n-1)d_u-\sum_{v\ne u}\mathbf1_{\{u\sim v\}}=(n-2)d_u$.
\end{proof}

\begin{lemma}\label{lem:dual}
	Let $X,Y$ be real symmetric matrices of the same order. Then $\norm X=\max_{-I\le Z\le I}\operatorname{tr}(ZX)$, where $Z$ is real symmetric. If $Z$ attains this maximum, then $Zx=\operatorname{sgn}(\lambda)x$ whenever $Xx=\lambda x$ and $\lambda\ne0$, and $ZX$ is positive semidefinite. Equality in $\norm{X+Y}\le\norm X+\norm Y$ holds precisely when a common such maximizer exists for $X$ and $Y$.
\end{lemma}

\begin{proof}
	Let $X$ be a real symmetric matrix of order $n$. Since $X$
	is symmetric, there exists an orthogonal matrix $U$ such that
	\[
	X=U\Lambda U^T,
	\qquad
	\Lambda=\operatorname{diag}(\lambda_1,\ldots,\lambda_n),
	\]
	where $\lambda_1,\ldots,\lambda_n$ are the eigenvalues of $X$.
	Let $Z$ be any real symmetric matrix satisfying
	$-I\le Z\le I$, and put $W=U^TZU$. Since $U$ is
	orthogonal, we have
	\[
	-I\le W\le I.
	\]
	In particular, the diagonal entries of $W$ satisfy
	$-1\le w_{ii}\le 1$ for every $i=1,\ldots,n$.
	Consequently,
	\begin{align*}
		\operatorname{tr}(ZX)
		&=\operatorname{tr}(U^TZU\Lambda)\\
		&=\sum_{i=1}^n\lambda_iw_{ii}\\
		&\le\sum_{i=1}^n|\lambda_i|
		=\operatorname{tr}|X|
		=\norm X.
	\end{align*}
	Thus,
	\[
	\max_{-I\le Z\le I}\operatorname{tr}(ZX)
	\le\norm X.
	\]
	To establish the reverse inequality, define
	\[
	Z_0=U\operatorname{diag}
	\bigl(\operatorname{sgn}(\lambda_1),\ldots,
	\operatorname{sgn}(\lambda_n)\bigr)U^T,
	\]
	where $\operatorname{sgn}(0)=0$.
	Clearly, $Z_0$ is real symmetric and
	$-I\le Z_0\le I$. Moreover,
	\begin{align*}
		\operatorname{tr}(Z_0X)
		&=\sum_{i=1}^n
		\lambda_i\operatorname{sgn}(\lambda_i)\\
		&=\sum_{i=1}^n|\lambda_i|
		=\norm X.
	\end{align*}
	Hence,
	\[
	\norm X=\max_{-I\le Z\le I}\operatorname{tr}(ZX).
	\]
	Next, suppose that a real symmetric contraction $Z$
	attains this maximum. Then
	\[
	\sum_{i=1}^n\lambda_iw_{ii}
	=\sum_{i=1}^n|\lambda_i|.
	\]
	Since
	$\lambda_iw_{ii}\le|\lambda_i|$ for every $i$,
	equality in the sum implies
	\[
	w_{ii}=\operatorname{sgn}(\lambda_i)
	\qquad\text{whenever }\lambda_i\ne0.
	\]
	We now show that this equality forces the corresponding
	eigenvectors to be eigenvectors of $Z$.
	Let $x$ be any unit eigenvector of $X$ corresponding
	to a nonzero eigenvalue $\lambda$. Since $Z$ attains
	the maximum, the preceding equality argument, applied
	to an orthonormal eigenbasis of $X$ containing $x$,
	gives
	\[
	x^TZx=\operatorname{sgn}(\lambda).
	\]
	If $\lambda>0$, then $x^TZx=1$. Since $Z\le I$,
	the matrix $I-Z$ is positive semidefinite and
	\[
	x^T(I-Z)x=0.
	\]
	For a positive semidefinite matrix $A$,
	the equality $x^TAx=0$ implies $Ax=0$.
	Therefore,
	\[
	(I-Z)x=0,
	\qquad\text{so that}\qquad Zx=x.
	\]
	Similarly, if $\lambda<0$, then $x^TZx=-1$.
	Since $I+Z$ is positive semidefinite,
	\[
	x^T(I+Z)x=0
	\]
	implies $(I+Z)x=0$, and hence $Zx=-x$.
	Thus, for every eigenvector $x$ of $X$ corresponding
	to a nonzero eigenvalue $\lambda$,
	\[
	Zx=\operatorname{sgn}(\lambda)x.
	\]
	It follows that $ZX$ and $|X|$ agree on every
	eigenspace of $X$ corresponding to a nonzero
	eigenvalue. They also agree on $\ker X$, since
	both operators vanish there. Consequently,
	\[
	ZX=|X|\ge0.
	\]
	Finally, we establish the equality condition in
	the triangle inequality. By the variational formula,
	\begin{align*}
		\norm{X+Y}
		&=\max_{-I\le Z\le I}
		\operatorname{tr}\bigl(Z(X+Y)\bigr)\\
		&=\max_{-I\le Z\le I}
		\bigl(\operatorname{tr}(ZX)
		+\operatorname{tr}(ZY)\bigr)\\
		&\le
		\max_{-I\le Z\le I}\operatorname{tr}(ZX)
		+\max_{-I\le Z\le I}\operatorname{tr}(ZY)\\
		&=\norm X+\norm Y.
	\end{align*}
	Suppose first that
	\[
	\norm{X+Y}=\norm X+\norm Y.
	\]
	Choose a symmetric contraction $Z_*$ attaining
	the maximum for $X+Y$. Then
	\[
	\operatorname{tr}(Z_*X)
	+\operatorname{tr}(Z_*Y)
	=\norm X+\norm Y.
	\]
	However, the variational formula gives
	\[
	\operatorname{tr}(Z_*X)\le\norm X,
	\qquad
	\operatorname{tr}(Z_*Y)\le\norm Y.
	\]
	Therefore, equality in their sum forces
	\[
	\operatorname{tr}(Z_*X)=\norm X,
	\qquad
	\operatorname{tr}(Z_*Y)=\norm Y.
	\]
	Hence $Z_*$ is a common maximizer for $X$ and $Y$.
	Conversely, suppose that there exists a real
	symmetric contraction $Z_*$ such that
	\[
	\operatorname{tr}(Z_*X)=\norm X,
	\qquad
	\operatorname{tr}(Z_*Y)=\norm Y.
	\]
	Then
	\begin{align*}
		\norm{X+Y}
		&\ge\operatorname{tr}\bigl(Z_*(X+Y)\bigr)\\
		&=\operatorname{tr}(Z_*X)
		+\operatorname{tr}(Z_*Y)\\
		&=\norm X+\norm Y.
	\end{align*}
	Combining this with the triangle inequality yields
	\[
	\norm{X+Y}=\norm X+\norm Y.
	\]
\end{proof}

\section{Main results}
\subsection{Deletion monotonicity of the trace norm}
We begin with a lower bound that applies to an arbitrary set of deleted vertices.
\begin{theorem}\label{thm:monotone}
For every $S\subseteq V(G)$,
\begin{equation}\label{eq:monotone}
 \T(G)-\T(G-S)\ge2\alpha e(S)\ge0.
\end{equation}
For a single vertex $v$, $L^T_\alpha(G;v)=0$ if and only if $v$ is isolated.
\end{theorem}
\begin{proof}
Let $T=V(G)\setminus S$ and partition $M=A_\alpha(G)$ accordingly. Set $R=\alpha\operatorname{diag}_{u\in T}|N_G(u)\cap S|$. Then
\[
 M[T,T]=A_\alpha(G-S)+R,\qquad
 \operatorname{tr}R=\alpha e(S,T).
\]
Pinching is contractive for the trace norm: for $U=I_S\oplus(-I_T)$,
\[
 \norm{M[S,S]}+\norm{M[T,T]}
 =\norm{\tfrac12(M+UMU)}\le\norm M.
\]
Since $\norm{M[S,S]}\ge\operatorname{tr}M[S,S]=\alpha\sum_{v\in S}d_v$ and $R$ is positive semidefinite,
\begin{align}
 \T(G-S)&\le\norm{M[T,T]}+\alpha e(S,T) \nonumber \\
 &\le \T(G) - \norm{M[S,S]} + \alpha e(S,T) \nonumber \\
 &\le \T(G) -\alpha\sum_{v\in S}d_v + \alpha e(S,T). \nonumber
\end{align}
Because $\sum_{v\in S}d_v=2e(S)+e(S,T)$, we obtain \eqref{eq:monotone}. The empty-set and full-set cases follow by the same argument, with empty blocks interpreted in the usual way.

The assertion about isolated vertices is immediate in one direction. Suppose $d=d_v>0$. For \(\alpha\ge \frac12\), we can write
\[
A_\alpha(G)=(2\alpha-1)D(G)+(1-\alpha)(D(G)+A(G)).
\]
Both \(D(G)\) and \(D(G)+A(G)=Q(G)\) are positive semidefinite. Therefore, \(A_\alpha(G)\) is also positive semidefinite. Similarly, \(A_\alpha(G-v)\) is positive semidefinite. Hence,
\[
\T(G)=\operatorname{tr}(A_\alpha(G))=2\alpha m,
\qquad
\T(G-v)=2\alpha(m-d_v).
\]
Therefore,
\[
L_\alpha^T(G;v)=2\alpha d_v>0.
\]
% the trace formula established below gives $L^T_\alpha(G;v)=2\alpha d>0$.
  At $\alpha=0$, strict positivity follows from \cite[Theorem 2]{ER}. It remains to consider the case \(0<\alpha<1/2\). Ordering the vertices so that \(v\) comes first, we write
  \begin{equation*}
  	M=A_\alpha(G)=
  	\begin{pmatrix}
  		\alpha d & \mathbf b^T\\
  		\mathbf b & P
  	\end{pmatrix},
  	\qquad
  	B=A_\alpha(G-v)=P-R,
  \end{equation*}
  where
  \begin{equation*}
  	\mathbf b=(1-\alpha)\mathbf 1_{N_G(v)}
  	\quad\text{and}\quad
  	R=\alpha\operatorname{diag}\mathbf 1_{N_G(v)}.
  \end{equation*}
  Here, \(\mathbf 1_{N_G(v)}\) denotes the indicator vector of \(N_G(v)\) in the retained vertex coordinates, and \(\operatorname{diag}\mathbf 1_{N_G(v)}\) denotes the diagonal matrix whose diagonal entries are \(1\) at the coordinates corresponding to \(N_G(v)\) and \(0\) elsewhere. If $\norm M=\norm B$, equality must hold in both
\[
 \norm M\ge\alpha d+\norm P,\qquad
 \norm B\le\norm P+\alpha d.
\]
By Lemma~\ref{lem:dual}, triangle equality in $B=P+(-R)$ gives a symmetric contraction $K$ maximizing both $P$ and $-R$. Since $\alpha>0$, $Ke_u=-e_u$ for every neighbour $u$ of $v$, and hence $K\mathbf b=-\mathbf b$. Set $Z=1\oplus K$. Then $\operatorname{tr}(ZM)=\alpha d+\norm P=\norm M$, so $ZM$ must be symmetric. Its upper-right block is $\mathbf b^T$, whereas its lower-left block is $K\mathbf b=-\mathbf b$, a contradiction because $\mathbf b\ne0$. Thus $L^T_\alpha(G;v)>0$.
\end{proof}

We next characterize equality in the general-set lower bound.
\begin{theorem}\label{thm:lowerclass}
Let $0<\alpha\le1$ and $S\subseteq V(G)$. Equality
\[
 \T(G)-\T(G-S)=2\alpha e(S)
\]
holds if and only if there are no edges between $S$ and $V(G)\setminus S$ and $A_\alpha(G[S])$ is positive semidefinite. Thus $S$ must be a union of components whose $A_\alpha$-matrices are positive semidefinite. For $\alpha\ge1/2$, the positive-semidefinite condition is automatic.
\end{theorem}
\begin{proof}
Use the notation of Theorem~\ref{thm:monotone}, and put $P=M[T,T]$ and $B=A_\alpha(G-S)=P-R$. The difference between the two sides of the lower bound is the sum of three nonnegative terms:
\begin{align*}
 &\norm M-\norm{M[S,S]}-\norm P,\\
 &\norm{M[S,S]}-\operatorname{tr}M[S,S],\\
 &\norm P+\operatorname{tr}R-\norm B.
\end{align*}
Consequently, equality forces all three terms to vanish. For $0<\alpha<1$, triangle equality for $B=P+(-R)$ gives a symmetric contraction $K$ maximizing $P$ and $-R$. Every vertex $u\in T$ incident with an edge from $S$ has a strictly positive diagonal entry of $R$, so $Ke_u=-e_u$. The second equality gives $M[S,S]\succeq0$. Hence $Z=I_S\oplus K$ satisfies
\[
 \operatorname{tr}(ZM)=\operatorname{tr}M[S,S]+\norm P=\norm M.
\]
By Lemma~\ref{lem:dual}, $ZM$ is symmetric. If $W=M[T,S]=(1-\alpha)A[T,S]$, then $KW=-W$, while the upper-right block of $ZM$ is $W^T$. Symmetry requires $KW=W$, whence $W=0$. Since $\alpha<1$, there are no edges between the two parts.

At $\alpha=1$, the trace-norm difference is $2(e(S)+e(S,T))$, so equality likewise forces $e(S,T)=0$. In either case, the vanishing cut identifies $M[S,S]$ with $A_\alpha(G[S])$. Conversely, when the cut is empty, block additivity gives $\T(G)-\T(G-S)=\norm{A_\alpha(G[S])}$, which equals its trace $2\alpha e(S)$ exactly when this block is positive semidefinite. Empty blocks cause no difficulty.
\end{proof}
\begin{corollary}\label{cor:strictcut}
If $G$ is connected, $0<\alpha\le1$, and $\varnothing\ne S\ne V(G)$, then
\[
 \T(G)-\T(G-S)>2\alpha e(S).
\]
\end{corollary}
\begin{proof}
Every nontrivial vertex partition of a connected graph has a nonempty cut. Apply Theorem~\ref{thm:lowerclass}.
\end{proof}

\subsection{Sharp bounds for the trace norm}
For $d\ge1$ define
\[
 F_\alpha(d)=
 \begin{cases}
 \alpha(d-1)+\sqrt{\alpha^2(d-1)^2+4(1-\alpha)^2d},&0\le\alpha<1/2,\\
 2\alpha d,&1/2\le\alpha\le1,
 \end{cases}
\]
and put $F_\alpha(0)=0$. We first present the sharp single-vertex bound.
\begin{theorem}\label{thm:single}
For every vertex $v$,
\begin{equation}\label{eq:single}
 0\le L^T_\alpha(G;v)\le F_\alpha(d_v).
\end{equation}
For $0\le\alpha<1/2$ and $d_v>0$, equality in the upper inequality $L^T_\alpha(G;v)\le F_\alpha(d_v)$ holds if and only if the component containing $v$ is a star with centre $v$. An isolated vertex has $L^T_\alpha(G;v)=0$. For $\alpha\ge1/2$, every vertex satisfies $L^T_\alpha(G;v)=2\alpha d_v$.
\end{theorem}
\begin{proof}
	Let \(d=d_v>0\). By Lemma \(\ref{lem:embed}\), we can write \(A_\alpha(G)=H_v+B_v\). Consider the \((d-1)\)-dimensional subspace
	\[
	L_v=\left\{x\in\mathbb R^n:x_u=0\text{ for }u\notin N_G(v),\ \sum_{u\sim v}x_u=0\right\}.
	\]
	For every \(x\in L_v\), we have \(H_vx=\alpha x\). Thus this subspace contributes \(d-1\) eigenvalues equal to \(\alpha\). The coordinates outside \(\{v\}\cup N_G(v)\) contribute \(n-d-1\) zero eigenvalues.
%Let $d=d_v>0$. $A_\alpha(G)=H_v+B_v$ as defined in Lemma \ref{lem:embed}. Consider the $(d-1)$-dimensional subspace
%$L_v=\{x: x_v=0, \sum_{u\sim v}x_u=0\}$. For every $x\in L_v$, we have $H_vx=\alpha x$. Hence $\alpha$ is an eigenvalue of $H_v$ with multiplicity $d-1$.
% On the leaf-contrast space, $H_v$ has eigenvalue $\alpha$ with multiplicity $d-1$.
On the span of $e_v$ and $w=d^{-1/2}\sum_{u\sim v}e_u$, its matrix is
\[
 K=\begin{pmatrix}\alpha d&(1-\alpha)\sqrt d\\(1-\alpha)\sqrt d&\alpha\end{pmatrix}.
\]
Its determinant is $d(2\alpha-1)$. For $\alpha<1/2$, its eigenvalues have opposite signs, with absolute sum $\sqrt{\alpha^2(d-1)^2+4(1-\alpha)^2d}$. For $\alpha\ge1/2$, it is positive semidefinite. Hence $\norm{H_v}=F_\alpha(d)$, and the triangle inequality, together with Theorem~\ref{thm:monotone}, proves \eqref{eq:single}.
For the equality characterization take $0\le\alpha<1/2$. Upper equality is equivalent to triangle equality for $A_\alpha(G)=H_v+B_v$. By Lemma~\ref{lem:dual} there is a symmetric contraction $Z$ with $ZH_v=|H_v|$ and $ZB_v$ positive semidefinite. The space spanned by $e_v,w$ is invariant under $Z$, and $Z$ equals $\operatorname{sgn}(K)$ there. Its off-diagonal entry is
\[
 \gamma=\frac{2(1-\alpha)\sqrt d}{\sqrt{\alpha^2(d-1)^2+4(1-\alpha)^2d}}>0.
\]
Thus $Ze_v=\eta e_v+\gamma w$ for some real $\eta$. Since $B_ve_v=0$, the $(v,v)$ entry of the positive semidefinite matrix $ZB_v$ is zero, so its entire $v$th row is zero. Therefore $w^TB_v=0$. All entries of $B_v$ are nonnegative, and all entries of $w$ on the neighbours are strictly positive. Each neighbour row of $B_v$ is consequently zero. No neighbour has an edge in $G-v$, and the component is the required star. The converse follows by additivity of trace norm over components and $\T(K_{1,d}-v)=0$.
Finally, $A_\alpha(G)=(1-\alpha)Q(G)+(2\alpha-1)D(G)$ is positive semidefinite for $\alpha\ge1/2$. Thus $\T(G)=2\alpha m$, proving the last assertion.
\end{proof}
\begin{remark}
At $\alpha=0$, Theorem~\ref{thm:single} recovers the bound and equality characterization of Espinal and Rada \cite{ER}. For $0<\alpha<1/2$ and $d>0$, the exact perturbation bound is strictly smaller than the separate triangle estimate $2\alpha d+2(1-\alpha)\sqrt d$, since
\[
 \sqrt{\alpha^2(d-1)^2+4(1-\alpha)^2d}
 <\alpha(d+1)+2(1-\alpha)\sqrt d.
\]
\end{remark}

Now let $S$ be an independent set of $k\ge1$ vertices sharing a common open neighbourhood $N$ of size $d\ge1$. Set
\[
 F_\alpha(k,d)=
 \begin{cases}
 \alpha(2kd-d-k)+\sqrt{\alpha^2(d-k)^2+4(1-\alpha)^2kd},&\alpha<1/2,\\
 2\alpha kd,&\alpha\ge1/2.
 \end{cases}
\]
\begin{theorem}\label{thm:twins}
Under these assumptions,
\[
 0\le\T(G)-\T(G-S)\le F_\alpha(k,d).
\]
For $0\le\alpha<1/2$, equality in the upper bound holds if and only if $S\cup N$ forms a component isomorphic to $K_{k,d}$, with $S$ as one part. For $\alpha\ge1/2$, the deletion difference is $2\alpha kd$ for every such graph. If $d=0$, the deletion difference is zero.
\end{theorem}
\begin{proof}
%Let $B_S$ be the embedded $A_\alpha(G-S)$. 
Let \(B_S\) denote the \(n\times n\) matrix obtained by embedding \(A_\alpha(G-S)\) into the coordinates corresponding to \(V(G)\setminus S\), with zero rows and columns indexed by \(S\).
Then $H_S=A_\alpha(G)-B_S$ is the embedded $A_\alpha(K_{k,d})$. Its spectrum consists of $\alpha d$ repeated $k-1$ times, $\alpha k$ repeated $d-1$ times, $0$ repeated $n-k-d$ times, and the two eigenvalues of
\[
 K_{k,d}=\begin{pmatrix}\alpha d&(1-\alpha)\sqrt{kd}\\(1-\alpha)\sqrt{kd}&\alpha k\end{pmatrix}.
\]
The determinant is $kd(2\alpha-1)$. Calculating its trace norm gives $F_\alpha(k,d)$ and proves the bound.
For upper equality with $0\le\alpha<1/2$, choose a common maximizer $Z$ for $H_S$ and $B_S$. Put $p=k^{-1/2}\sum_{s\in S}e_s$ and $q=d^{-1/2}\sum_{u\in N}e_u$. On their span, $Z=\operatorname{sgn}(K_{k,d})$, so $Zp=\eta p+\gamma q$, where
\[
 \gamma=\frac{2(1-\alpha)\sqrt{kd}}{\sqrt{\alpha^2(d-k)^2+4(1-\alpha)^2kd}}>0.
\] Since $B_Sp=0$, positive semidefiniteness of $ZB_S$ implies $p^TZB_S=0$, whence $q^TB_S=0$. Nonnegativity of $B_S$ forces every row indexed by $N$ to vanish. Thus $N$ has no edges within itself or outside $S\cup N$, giving the stated component. The converse is immediate. Nonnegativity follows from Theorem~\ref{thm:monotone}. The trace formula proves the result for $\alpha\ge1/2$.
\end{proof}

\subsection{A deletion bound for the standard centred energy}
The next theorem applies to the centred definition used in the $A_\alpha$-energy literature and remains meaningful throughout $[0,1]$.
\begin{theorem}\label{thm:center}
Let $n\ge2$, $v$ have degree $d$, and set
\[
 c=\frac{2\alpha m}{n},\quad h=\frac{2\alpha(m-d)}{n-1},\quad s=h-c.
\]
For $d\ge1$ put $x=\alpha d-c$, $y=\alpha+s$, and
\begin{equation}\label{eq:centerbound}
 R_\alpha(n,m,d)=(d-1)|y|+(n-d-1)|s|
 +\max\left\{|x+y|,\sqrt{(x-y)^2+4(1-\alpha)^2d}\right\}.
\end{equation}
For $d=0$ put $R_\alpha(n,m,0)=|c|+(n-1)|s|$. Then
\[
 |L^C_\alpha(G;v)|\le R_\alpha(n,m,d).
\]
The bound is attained for every $\alpha\in[0,1]$ when $G=K_{1,d}$ and $v$ is its centre. At $\alpha=0$, it reduces to $2\sqrt d$.
\end{theorem}
\begin{proof}
Embed $C_\alpha(G-v)$ as $\widetilde C_v$ with zero at $v$. The difference $\Delta=C_\alpha(G)-\widetilde C_v$ has diagonal $x$ at $v$, $y$ at neighbours, and $s$ elsewhere. Its only off-diagonal entries are $1-\alpha$ along the $d$ incident edges. Its eigenvalues are $y$ with multiplicity $d-1$, $s$ with multiplicity $n-d-1$, and those of
\[
 \begin{pmatrix}x&(1-\alpha)\sqrt d\\(1-\alpha)\sqrt d&y\end{pmatrix}.
\]
For a symmetric two-by-two matrix with diagonal $x,y$ and off-diagonal $b$, the sum of the absolute eigenvalues is $\max\{|x+y|,\sqrt{(x-y)^2+4b^2}\}$. This proves \eqref{eq:centerbound} as an exact expression for $\norm\Delta$. The reverse triangle inequality yields the result. If $d=0$, the difference is diagonal with the stated norm. For a star centre, $G-v$ is edgeless and $\widetilde C_v=0$, proving attainment.
\end{proof}
\begin{remark}
The bound is sharp on the star-centre family. Theorem~\ref{thm:centerclass} below classifies its connected equality graphs. Pointwise optimality for every feasible triple $(n,m,d)$ is not asserted.
\end{remark}

The same calculation applies to independent twins and gives a bound determined by four graph parameters.
\begin{theorem}\label{thm:centertwins}
Let $S$ be an independent set of $k\ge1$ vertices with a common neighbourhood of size $d\ge1$, where $n>k$. Put
\[
 c=\frac{2\alpha m}{n},\qquad h=\frac{2\alpha(m-kd)}{n-k},\qquad
 s=h-c,\qquad x=\alpha d-c,\qquad y=\alpha k+s.
\]
Then
\begin{align}\label{eq:centertwins}
 |\E(G)-\E(G-S)|\le{}&(k-1)|x|+(d-1)|y|+(n-k-d)|s|\nonumber\\
 &+\max\left\{|x+y|,\sqrt{(x-y)^2+4(1-\alpha)^2kd}\right\}.
\end{align}
For every $\alpha\in[0,1]$, equality holds when $G=K_{k,d}$ and $S$ is the part of size $k$.
\end{theorem}
\begin{proof}
Embed $C_\alpha(G-S)$ with zero rows and columns on $S$. The difference from $C_\alpha(G)$ has constant diagonal $x$ on $S$, $y$ on the common neighbourhood, and $s$ elsewhere. Its only off-diagonal entries form $(1-\alpha)A(K_{k,d})$. The two contrast spaces contribute eigenvalues $x$ and $y$ with multiplicities $k-1$ and $d-1$, respectively. The remaining exterior space contributes $s$ with multiplicity $n-k-d$. On the normalized indicator vectors of the two parts, the difference has matrix
\[
 \begin{pmatrix}x&(1-\alpha)\sqrt{kd}\\(1-\alpha)\sqrt{kd}&y\end{pmatrix}.
\]
Its trace norm gives the right-hand side of \eqref{eq:centertwins}. Apply the reverse triangle inequality. For $K_{k,d}$, the retained graph is edgeless, so its centred matrix is zero and equality follows.
\end{proof}

We now give a complete structural equality classification for the centred bounds on connected graphs. It includes singular perturbations and applies to the absolute deletion difference, not only to its upper branch.
\begin{theorem}\label{thm:centerclass}
Let $G$ be connected, let $0\le\alpha<1$, and let $S$ be an independent set of $k\ge1$ vertices with a common neighbourhood $N$ of size $d\ge1$. Equality in \eqref{eq:centertwins} holds if and only if
\[
 G\cong K_{k,d},\qquad S\text{ is the part of size }k.
\]
In particular, for a vertex $v$ of a connected graph of order at least two,
\[
 |L^C_\alpha(G;v)|=R_\alpha(n,m,d_v)
\]
holds if and only if $G$ is a star and $v$ is its centre. Every such equality lies on the upper branch: the original centred energy exceeds the retained energy.
\end{theorem}
\begin{proof}
Let $A=C_\alpha(G)$, let $B$ be the embedded $C_\alpha(G-S)$, and put $\Delta=A-B$. Retain $x,y,s$ from Theorem~\ref{thm:centertwins}, and set
\[
 p=k^{-1/2}\sum_{u\in S}e_u,\qquad
 q=d^{-1/2}\sum_{u\in N}e_u,\qquad
 W=V(G)\setminus(S\cup N).
\]
On the span of $p,q$, the matrix $\Delta$ is
\[
 K=\begin{pmatrix}x&b\\b&y\end{pmatrix},\qquad
 b=(1-\alpha)\sqrt{kd}>0.
\]
Both $A$ and $B$ are traceless, so $\Delta$ is traceless as well.
First exclude lower-branch equality $\norm B-\norm A=\norm\Delta$. This would give triangle equality in $B=A+(-\Delta)$ and a common maximizer $Z$ with
\[
 ZB=|A|+|\Delta|\succeq0.
\]
Since $Bp=0$, it follows that $p^TZBp=0$. Thus, $p^T|A|p+p^T|\Delta|p=0$ implies that $p^T|\Delta|p=0$. Hence, $\Delta p=0$.
%taking the quadratic form at $p$ forces $p^T|\Delta|p=0$, hence $\Delta p=0$.
But $\Delta p=x p+bq\ne0$. Thus lower-branch equality is impossible.
Suppose upper equality holds. Triangle equality in $A=B+\Delta$ gives a symmetric contraction $Z$ maximizing $B$ and $\Delta$. In particular, $ZB=|B|\succeq0$ and $Z\Delta=|\Delta|$. Since $Bp=0$, the positive-semidefinite matrix $ZB$ annihilates $p$, and
\begin{equation}\label{eq:pZB}
 p^TZB=0.
\end{equation}
If $K$ is indefinite, its two eigenvalues are nonzero. Thus its span is invariant under $Z$, and
\[
 Zp=\eta p+\gamma q,\qquad
 \gamma=\frac{2b}{\sqrt{(x-y)^2+4b^2}}>0.
\]
Equation~\eqref{eq:pZB} then yields $q^TB=0$. For a vertex $w\in W$, the corresponding entry is
\[
 (q^TB)_w=\frac{1-\alpha}{\sqrt d}\,|N_G(w)\cap N|.
\]
It follows that there are no edges between $N$ and $W$. There are also no edges from $S$ to $W$ by hypothesis. Connectivity therefore implies $W=\varnothing$. If $r_u$ is the degree of $u$ in $G[N]$, the entry of $q^TB$ at $u\in N$ is $(r_u-h)/\sqrt d$. Hence $G[N]$ is regular of degree $r=h$. Since $G-S=G[N]$, its centring constant is $h=\alpha r$. As $\alpha<1$, this forces $r=0$, and $G=K_{k,d}$.
It remains to consider a definite or singular semidefinite $K$. Write $\sigma\in\{-1,1\}$ for its nonzero eigenvalue sign. Because $b>0$, both $x$ and $y$ are nonzero with sign $\sigma$. The contrast spaces on $S$ and $N$ also have this sign. The identity
\[
 kx+dy+|W|s=\operatorname{tr}\Delta=0
\]
shows that $W\ne\varnothing$ and $s$ has sign $-\sigma$.
If $K$ is definite, $Z=\sigma I$ on $S\cup N$ and $Z=-\sigma I$ on $W$. Since $ZB$ is symmetric, $ZB=BZ$, so the off-diagonal block $B[N,W]$ is zero. This disconnects $G$, a contradiction.
If $K$ is singular, its zero eigenspace is one-dimensional and all other eigenspaces of $\Delta$ have nonzero eigenvalues. Thus $Z$ acts by a scalar $t\in[-1,1]$ on the zero eigenspace of $K$ and by $\sigma$ on its nonzero eigenspace. In the basis $p,q$, write $Zp=\eta p+\gamma q$. If $\gamma\ne0$, \eqref{eq:pZB} again gives $q^TB=0$, which disconnects $G$ because $W\ne\varnothing$. If $\gamma=0$, then $t=\sigma$: the nonzero spectral projection of $K$ has a nonzero off-diagonal entry because $b>0$. Therefore $Z=\sigma I$ on $S\cup N$, and the preceding commutation argument again disconnects $G$.
The only remaining case is $G=K_{k,d}$. Conversely, for that graph $B=0$, so upper equality follows directly. Taking $k=1$ gives the single-vertex statement.
\end{proof}
\begin{remark}
Connectivity is essential: $K_{k,d}$ together with isolated vertices also attains the centred bound when its part $S$ is deleted, because the retained graph is edgeless. The theorem classifies all connected extremal graphs, rather than asserting the same classification for arbitrary disconnected graphs.
\end{remark}

The endpoint $\alpha=1$ has a different equality mechanism. It admits an exact description in terms of degrees alone.
\begin{proposition}\label{prop:degreeendpoint}
Let $\varnothing\ne S\ne V(G)$, let $H=G-S$, and define
\[
 a_u=d_G(u)-\frac{2m(G)}{|V(G)|},\qquad
 b_u=\begin{cases}
 0,&u\in S,\\
 d_H(u)-\dfrac{2m(H)}{|V(H)|},&u\notin S,
 \end{cases}
 \qquad \delta_u=a_u-b_u.
\]
Then
\[
 |\mathcal E_1(G)-\mathcal E_1(H)|\le\sum_u|\delta_u|.
\]
Upper-branch equality holds if and only if $b_u\delta_u\ge0$ for every $u$. Lower-branch equality holds if and only if $a_u\delta_u\le0$ for every $u$. Thus equality in the absolute bound holds precisely when at least one of these two systems is satisfied. For single vertices and independent twins, the right-hand side is exactly the perturbation bound stated above at $\alpha=1$.
\end{proposition}
\begin{proof}
At $\alpha=1$, the three matrices are diagonal with entries $a_u,b_u,\delta_u$, respectively. Equality in the sum of the scalar triangle inequalities $|b_u+\delta_u|\le|b_u|+|\delta_u|$ requires and is implied by $b_u\delta_u\ge0$ at every coordinate. Apply the same argument to $b_u=a_u-\delta_u$ for the lower branch.
\end{proof}

\subsection{Bounds for total local energy}

For $\bar d=2m/n$ define the degree deviation
\[
 \operatorname{dev}(G)=\sum_{v\in V(G)}|d_v-\bar d|.
\]
\begin{theorem}\label{thm:total}
For every graph of order $n\ge2$,
\begin{align}
 0\le\ell^T_\alpha(G)&\le2\T(G),\label{eq:totalT}\\
 \ell^C_\alpha(G)&\le2\E(G)+\frac{2\alpha}{n-1}\operatorname{dev}(G).\label{eq:totalC}
\end{align}
For $\alpha\ge1/2$, the upper bound in \eqref{eq:totalT} is attained for every graph: $\ell^T_\alpha(G)=2\T(G)$. The second reduces to $\ell^C_\alpha(G)\le2\E(G)$ for regular graphs or $\alpha=0$. No constant $q<1$ can replace the additional term in \eqref{eq:totalC} by $q\,2\alpha\operatorname{dev}(G)/(n-1)$ uniformly over all graphs and all $\alpha\in[0,1]$. Equality with positive deviation occurs at $\alpha=1$ for $K_2\cup K_1$.
\end{theorem}
\begin{proof}
By \eqref{eq:sum}, $(n-2)\T(G)\le\sum_v\T(G-v)$, proving the upper bound in \eqref{eq:totalT}; its lower bound follows from Theorem~\ref{thm:monotone}. For the centred matrices, at a fixed retained vertex $u$,
\[
 \sum_{v\ne u}c_{G-v}
 =\frac{2\alpha}{n-1}\bigl((n-3)m+d_u\bigr).
\]
Therefore the exact embedded summation identity is
\begin{equation}\label{eq:centeridentity}
 \sum_v\widetilde C_v=(n-2)C_\alpha(G)
 -\frac{2\alpha}{n-1}\bigl(D(G)-\bar dI_n\bigr).
\end{equation}
The triangle inequality gives
\[
 (n-2)\E(G)\le\sum_v\E(G-v)
 +\frac{2\alpha}{n-1}\operatorname{dev}(G),
\]
Rearranging yields \eqref{eq:totalC}. The trace formula proves the stated trace-norm equality.
For $G=K_2\cup K_1$ and $\alpha=1$, the degrees are $(1,1,0)$, $\E(G)=4/3$, and each deletion leaves a regular graph, hence energy zero. Thus $\ell^C_1(G)=4$, $\operatorname{dev}(G)=4/3$, and both sides of \eqref{eq:totalC} equal $4$. If the additional term were replaced by $q$ times itself, with $q<1$, the right-hand side would become
\[
\frac{8}{3}+q\frac{4}{3}<4,
\]
contradicting the actual value $\ell_1^C(G)=4$. This proves the stated uniform necessity of the degree deviation term.
\end{proof}
\begin{remark}
The sharpness assertion is uniform over the full parameter range. It does not assert that the same coefficient is optimal for every fixed $0<\alpha<1$ and every fixed order.
\end{remark}

\begin{corollary}\label{cor:energyonly}
For $n\ge2$,
\[
 \ell^C_\alpha(G)\le\left(2+\frac{2}{n-1}\right)\E(G).
\]
The multiplicative coefficient is attained for $\alpha=1$ and $G=K_2\cup K_1$.
\end{corollary}
\begin{proof}
Diagonal pinching of $C_\alpha(G)$ gives $\alpha\operatorname{dev}(G)\le\E(G)$. Apply Theorem~\ref{thm:total}. The example in its proof gives equality.
\end{proof}
\begin{remark}
 Theorem~\ref{thm:total} extends the Laplacian degree-deviation estimate of Moon and Park \cite[Theorem 2]{MP} to the present setting. If $G$ is bipartite, $Q(G)$ and $L(G)$ are diagonally similar. Since $A_{1/2}(G)=Q(G)/2$, one has $\mathcal E_{1/2}(G)=LE(G)/2$, and the same relation holds for every induced deletion. Multiplying \eqref{eq:totalC} by two therefore recovers their total local Laplacian-energy bound on bipartite graphs.
\end{remark}

The total centred-energy estimate has a further structural restriction on its equality cases.
\begin{theorem}\label{thm:totalstrict}
If $G$ is connected and nonregular and $0<\alpha<1$, then
\[
 \ell^C_\alpha(G)<2\E(G)+\frac{2\alpha}{n-1}\operatorname{dev}(G).
\]
Consequently, any connected equality graph in this parameter range must be regular.
\end{theorem}
\begin{proof}
Such a graph has $n\ge3$. Set $R=2\alpha(D-\bar d I)/(n-1)$, so \eqref{eq:centeridentity} reads $(n-2)C_\alpha(G)=\sum_v\widetilde C_v+R$. Equality in the estimate would require a symmetric contraction $Z$ maximizing every $\widetilde C_v$ and $R$ simultaneously. Choose a vertex $w$ with $d_w<\bar d$. Since the corresponding diagonal entry of $R$ is negative, $Ze_w=-e_w$.

There is a vertex $v\ne w$ for which $G-v$ is connected: choose a leaf other than $w$ of a spanning tree of $G$. The graph $G-v$ has at least two vertices. Its centred matrix is nonzero and traceless, so its largest eigenvalue is strictly positive. Since $1-\alpha>0$, a Perron eigenvector of $A_\alpha(G-v)$ is strictly positive and is also an eigenvector for the largest eigenvalue of its centred matrix. Embed this vector as $z$ with zero at $v$. The maximizing property gives $Zz=z$, whereas $z_w>0$. Symmetry of $Z$ now implies
\[
 z_w=e_w^TZz=(Ze_w)^Tz=-z_w,
\]
a contradiction.
\end{proof}

At $\alpha=1$, equality in the total estimate can be decided completely from the degree sequence and adjacency incidences.
\begin{theorem}\label{thm:totalendpoint}
Let $n\ge2$ and put
\[
 a_u=d_u-\bar d,\qquad
 b_{uv}=d_u-\mathbf1_{\{u\sim v\}}-\frac{2(m-d_v)}{n-1};\quad u\ne v,\ \mathbf1_{\{u\sim v\}}=\begin{cases}
 	1,&u\sim v, \\
 	0,&u\nsim v. \end{cases}
\]
Equality
\[
 \ell^C_1(G)=2\mathcal E_1(G)+\frac{2}{n-1}\operatorname{dev}(G)
\]
holds if and only if, for each vertex $u$, the following conditions hold:
\begin{enumerate}
 \item if $a_u\ne0$, then $a_u b_{uv}\ge0$ for all $v\ne u$;
 \item if $a_u=0$, then $b_{uv}=0$ for all $v\ne u$.
\end{enumerate}
In particular, among regular graphs, equality holds precisely for complete graphs and edgeless graphs.
\end{theorem}
\begin{proof}
All matrices in \eqref{eq:centeridentity} are diagonal when $\alpha =1$. The $u$th coordinate of the decomposition is
\[
 (n-2)a_u=\sum_{v\ne u}b_{uv}+\frac{2}{n-1}a_u.
\]
The total bound is the sum of the scalar triangle inequalities at these coordinates. If $a_u\ne0$, the additional summand has the sign of $a_u$, and equality requires every other summand to have that same weak sign. If $a_u=0$, the sum is zero and equality requires every summand to vanish. This proves the criterion, including $n=2$.

For a regular graph of degree $r$, $a_u=0$ for every $u$, so equality requires every vertex-deleted graph to be regular. If $0<r<n-1$, each vertex $v$ has both a neighbour and a nonneighbour among the retained vertices. Their retained degrees are $r-1$ and $r$, respectively, so $G-v$ is not regular. Thus only $r=0$ and $r=n-1$ are possible. Both edgeless and complete graphs satisfy the criterion.
\end{proof}

\subsection{Why adjacency-energy properties do not transfer verbatim}
\begin{proposition}\label{prop:obstructions}
Centred deletion-based local energy can be negative, can be positive at an isolated vertex, and can violate $\ell^C_\alpha(G)\le2\E(G)$.
\end{proposition}
\begin{proof}
At $\alpha=1$, $\E(G)=\sum_v|d_v-\bar d|$. For $G=C_4$, $\E(G)=0$, whereas $G-v=P_3$ has degree sequence $(1,2,1)$ and energy $4/3$. Thus $L^C_1(C_4;v)=-4/3$ for every vertex. For $K_2\cup K_1$, deleting the isolated vertex leaves the regular graph $K_2$, so the isolated vertex has local energy $4/3$. The same graph gives $\ell^C_1=4>8/3=2\E$, as calculated above.
\end{proof}
Because energies depend continuously on $\alpha$, the strict negative change for $C_4$ and the strict total-energy violation for $K_2\cup K_1$ also persist on nonempty intervals immediately below $\alpha=1$. They are not merely artifacts of an isolated endpoint.

For completeness, positive semidefiniteness gives the exact trace-norm formula $\T(G)=2\alpha m$ for $\alpha\ge1/2$. The graph-dependent threshold can be expressed through normalized adjacency. If $G$ has at least one edge, then on the nonisolated vertices,
\[
 A_\alpha(G)=D^{1/2}\bigl(\alpha I+(1-\alpha)D^{-1/2}AD^{-1/2}\bigr)D^{1/2}.
\]
Consequently, if $r_{\min}$ is the smallest normalized-adjacency eigenvalue, the threshold is $-r_{\min}/(1-r_{\min})$. These are supporting facts from the positive-semidefiniteness theory; see \cite{NR}. They do not imply that the centred energy is constant or trivial in that range.

\section{Numerical examples and reproducibility}
The following values were calculated by direct symmetric eigensolvers. Vertex labels for $P_4$ are consecutive along the path, and the deleted vertex is internal. In $K_{2,3}$ the deleted vertex belongs to the part of size two.
\begin{center}
\begin{tabular}{lrrrr}
\toprule
Graph and deleted vertex&$\alpha$&$L^T$&$L^C$&$R_\alpha$\\\midrule
$K_{1,4}$, centre&0.25&3.842329&3.542329&3.542329\\
$K_{1,4}$, centre&0.75&6.000000&3.812214&3.812214\\
$K_{2,3}$, part of size 2&0.40&2.739374&0.899374&2.528061\\
$K_{2,3}$, part of size 2&1.00&6.000000&$-0.600000$&1.800000\\
$K_3$, any vertex&0.25&1.000000&1.500000&2.121320\\
$P_4$, internal vertex&0.40&1.816553&1.311936&2.102286\\\bottomrule
\end{tabular}
\end{center}
The equality classifications distinguish component separation from positivity of the retained spectral quantities. For example, take $G=K_3\cup K_2$ and let $S$ be the triangle component. The cut is empty, but
\[
 A_\alpha(K_3)\succeq0\quad\Longleftrightarrow\quad\alpha\ge\frac13.
\]
Thus equality in Theorem~\ref{thm:lowerclass} holds at $\alpha=0.4$ and fails at $\alpha=0.25$. At the latter value the deletion difference is $5/2$, whereas the lower bound is $3/2$.

Singular perturbations are also covered by the centred classification. If $G=P_4$ and $v$ is an internal vertex, then $x=\alpha/2$ and $y=\alpha/6$. The quotient in Theorem~\ref{thm:center} is singular at
\[
 \alpha=\frac{2\sqrt6}{1+2\sqrt6}.
\]
Nevertheless, the centred deletion bound is strict because $P_4$ is not a star centred at $v$. As an exact example of Theorem~\ref{thm:totalstrict}, $G=P_3$ and $\alpha=1/2$ give
\[
 \ell^C_{1/2}(P_3)=3
 <4=2\mathcal E_{1/2}(P_3)+\frac{1}{2}\operatorname{dev}(P_3).
\]

All tabulated values were independently recalculated using a symmetric eigensolver. Additional consistency checks enumerate every labelled simple graph of orders two through five with
\[
 \alpha\in\{0,0.1,0.25,0.4,0.49,0.5,0.75,0.9,1\}.
\]
These checks comprise 48,636 single-vertex cases and 9,882 graph--parameter cases. They test the single-vertex trace-norm and centred-energy bounds, the connected single-vertex centred equality classification, the total local-energy bounds, and strictness of the total centred bound for connected nonregular graphs when $0<\alpha<1$. Further checks cover 294,894 nonempty vertex-set deletion cases for the general-set lower bound (with its equality criterion tested for $\alpha>0$), and 51,408 independent-twin cases with positive common-neighbourhood size for the trace-norm and centred-energy bounds and the connected centred equality classification. No violations were found using absolute tolerance $10^{-8}$. These computations are consistency tests rather than proofs; the arguments above establish the results.

\section{Conclusion}
The trace norm and the centred energy of $A_\alpha(G)$ have different responses to vertex deletion. The trace norm is monotone under induced-subgraph deletion, and its sharp upper deletion bounds admit structural equality characterizations. In contrast, the centred local energy may have either sign, and its total-energy estimate requires a degree-deviation term. Exact perturbation spectra give explicit bounds for single vertices and independent twins in both settings.

The centred deletion bounds now have a complete equality classification on connected graphs for $0\le\alpha<1$. The general-set trace-norm lower bound is classified for $0<\alpha\le1$, and total centred-energy equality has a degree criterion at $\alpha=1$. Further questions concern the disconnected centred equality graphs for $0<\alpha<1$, equality among regular connected graphs in the total centred bound for that range, the optimal coefficient in the total centred-energy bound for a fixed $\alpha$ and order, and equality in $\ell^T_\alpha(G)=2\T(G)$ for $0<\alpha<1/2$. The last question has additional equality families beyond matching components. For example, if $G=K_n$ with $n\ge4$ and $1/(n-1)\le\alpha<1/2$, both $A_\alpha(G)$ and every $A_\alpha(G-v)$ are positive semidefinite, and the total-energy equality follows from their traces.

\end{document}